\RequirePackage{fix-cm}
\documentclass[12pt]{article}
\pdfoutput=1
\usepackage{geometry}
\usepackage[utf8]{inputenc}
\usepackage{graphicx}
\usepackage[caption]{subfig}
\usepackage{mwe}
\usepackage{color}
\usepackage{amsmath}
\usepackage{amssymb}
\usepackage{hyperref}
\usepackage{algorithm}
\usepackage{algorithmic}
\newcommand{\trialX}{\mathcal{X}}
\newcommand{\testY}{\mathcal{Y}}

\newcommand{\trialXd}{\mathcal{X}_{\delta}}
\newcommand{\testYd}{\mathcal{Y}_{\delta}}
\newcommand{\parset}{\mathcal{D}}
\newcommand\restr[2]{\ensuremath{\left.#1\right|_{#2}}}
\usepackage{authblk}
\newtheorem{prop}{Proposition}[section]

\newenvironment{proof}{\paragraph{Proof:}}{\hfill$\square$}

\title{A space-time certified reduced basis method for quasilinear parabolic partial differential equations}
\author[1]{Michael Hinze\thanks{hinze@uni-koblenz.de}}
\author[1]{Denis Korolev\thanks{korolev@uni-koblenz.de}}
\date{}

\affil[]{University of Koblenz-Landau, Mathematical Institute}

\begin{document}

\maketitle

\begin{abstract}
 In this paper, we propose a certified reduced basis  (RB) method for \textit{quasilinear} parabolic problems. The method is based on \textit{a space-time} variational formulation. We provide a residual-based \textit{a-posteriori} error bound for a space-time formulation and the corresponding efficiently computable estimator for the \textit{certification} of the method. We use the Empirical Interpolation method (EIM) to guarantee the efficient \textit{offline-online} computational procedure. The error of the EIM method is then rigorously incorporated into the certification procedure. The  Petrov-Galerkin finite element discretization allows to benefit from the Crank-Nicolson interpretation of the discrete problem and to use a POD-Greedy approach to construct the reduced-basis spaces of small dimensions. It computes the reduced basis solution in a time-marching framework while the  RB approximation error in a space-time norm is controlled by the estimator. Therefore we combine a POD-Greedy approximation with a space-time Galerkin method. 
\end{abstract}

\section{Introduction}
\label{intro}
The certified reduced basis method is known as an efficient method for model order reduction of parametrized partial differential equations (see, e.g. \cite{haasdonk2017reduced,quarteroni2015reduced} , where also the terminology used in the present article is well-explained). The efficiency comes from the use of the Greedy search algorithm in the basis construction for the numerical approximation of the problem and a-posteriori control of the approximation error. The later serves not only for rigorous certification of the method, but also as the selection criterion in the Greedy selection process. This process provides incrementally better bases for the approximation and further significant speed-up in multi-query numerical simulations - relevant, for example, in the design, optimization and control contexts, through the use of RB surrogate models.

The reduced basis method was successfully applied to linear \cite{grepl2005posteriori,steih2012space,urban2014improved} and nonlinear \cite{grepl2012certified,yano2014space,yano2014Boussinesq} parabolic problems, where the spatial differential operator is coercive \cite{grepl2012certified,steih2012space} or inf-sup stable \cite{urban2014improved,yano2014space,yano2014Boussinesq}.  In general, there are two approaches for the reduced basis methods applied to unsteady problems: (1) first discretize, then estimate and reduce, (2) first estimate, then discretize and reduce. The approach (1) \cite{grepl2005posteriori,haasdonk2008reduced,grepl2012certified} is based on a time-marching problem in the offline phase and the error bounds or indicators are then stem from the structure of the discrete problem. The POD-Greedy procedure \cite{haasdonk2008reduced} is commonly used to construct the reduced-basis spaces and the Empirical Interpolation Method (EIM) \cite{barrault2004empirical,grepl2007efficient,maday2009general} is used to treat non-affine and nonlinear problems \cite{grepl2012certified}. The approach (2) starts from a weak space-time variational formulation (see, e.g. \cite{steih2012space,urban2014improved,yano2014space,yano2014Boussinesq}). The error bounds are then derived in the appropriate Bochner spaces with respect to the natural space-time norms. In this approach time is treated as a variable and thus it resembles the reduced-basis setting for elliptic problems \cite{rozza2007reduced}. The reduced-basis space is consequently constructed in the offline phase out of the space-time snapshots, obtained, for example, with the related Petrov-Galerkin discrete scheme.  However, the appropriate choice of the discrete spaces in the Petrov-Galerkin scheme results in a time-marching interpretation (see, e.g. \cite{urban2014improved,yano2014space}) of the discrete problem. In this way, the time-marching procedure allows to use the standard POD-Greedy approximation and to treat time as the parameter, which leads to the reduced-basis time-marching problem, but the error certification is accomplished  with the natural space-time norm error bound. We refer to \cite{glas2017two} for the detailed overview and comparison of these two approaches in the context of linear parabolic equations.

In this paper we treat quasilinear parabolic problems with the approach (2). We propose our $L^{2}(0,T;V)$ a-posteriori error bound, based on the space-time variational formulation of quasilinear parabolic PDEs with strongly monotone differential operators. We provide a Crank-Nicolson time-stepping interpretation of the discrete Petrov-Galerkin problem and consequently use the POD-Greedy procedure to construct the reduced-basis spaces of small dimension. 

A time-marching interpretation also allows to treat the nonlinearity with the EIM in order to have offline-online decomposition for our problem available. Moreover, the parameter separability in time, achieved with the EIM, leads to a significant speed-up factor in the computational procedure. The error of the EIM is then also incorporated in the error bound. 

Our work is motivated by the structure of the magnetoquasistatic approximation of Maxwell's equations (the eddy-current model). This equation finds its place in important applications, such as the computation of magnetic fields in the presence of eddy currents in electrical machines \cite{salon1995finite}. The development of fast and accurate simulation methods for such problems is of great importance in the optimization and design of electrical machines and other devices \cite{alla2019certified,ion2018robust}. Therefore there is a demand for reduced order models (see, e.g. \cite{Stykel2017}) of this quasilinear PDE, which can be further used as surrogates in the optimization procedure. Our approach is applicable to the 2-D magnetoquasistatic problem as well, and we present according numerical results.

\section{Space-Time Truth Solution}
\label{sec:2}
In this section we consider a space-time variational formulation of quasilinear parabolic partial differential equations, which we denote as the \textit{exact problem}. The corresponding discrete Petrov-Galerkin  approximation is called the \textit{truth} problem, as it is common in the RB setting. We assume that the solution to the exact problem can be approximated arbitrarily well by the discrete solution of the truth problem. We then neglect the corresponding approximation error. 

\subsection{\textbf{Space-Time formulation}}
\label{sec:2.1}
Let $\Omega \subset \mathbb{R}^{d}$ be the spatial domain and $\mu \in \parset \subset \mathbb{R}^{p}$, where $\parset$ is a compact parameter set. Let  $V \subset H^{1}(\Omega) $  be a separable Hilbert space and $H:=L^{2}(\Omega)$. We denote by $ \langle \cdot,\cdot \rangle_{{V}}$, $\langle \cdot,\cdot \rangle_{H} $ and $ \lVert \cdot \rVert_{V}$, $\lVert  \cdot \rVert_{H}$ corresponding inner products and induced norms, respectively. To $V$ and $H$ we associate the Gelfand triple $V\hookrightarrow H \hookrightarrow V'$
 with duality pairing $\langle \cdot,\cdot \rangle_{V'V}$. The norm of $l\in V'$ is defined by $\Vert l\rVert_{V'}:=\underset{\psi \in V,\lVert \psi \rVert_{V}\neq 0}{\sup}\langle l,\psi \rangle_{V'V}/\lVert \psi \rVert_{V}$. We consider a parametrized quasilinear, bounded differential operator $A:V \times \parset \rightarrow V'$ with induced quasilinear form
\begin{align}\label{2.1.1}
\langle A(u,\mu),v \rangle_{V'V}:=a[u](u,v;\mu)=\int_{\Omega} \nu(u(x);\mu)\nabla u \cdot \nabla v\ dx,
\end{align} 
where the nonlinearity satisfies $\nu(\cdot;\mu) \in C^{1}(\mathbb{R})$. We assume that the forms \eqref{2.1.1} are strongly monotone on $V$  with monotonicity constants $m_{a}(\mu)>0$, i.e.
\begin{align}\label{2.1.2}
    a[v](v,v-w;\mu)-a[w](w,v-w;\mu) \geq m_{a}(\mu) \lVert v-w \rVert_{V}^{2}  \quad  \forall \, v,w\in V,
\end{align}
and Lipschitz continuous on $V$ with Lipschitz constants $L_{a}(\mu)>0$, i.e.
\begin{align}\label{2.1.3}
|a[u](u,v;\mu)-a[w](w,v;\mu)|\leq L_{a}(\mu) \lVert u-w \rVert_{V} \lVert v \rVert_{V} \quad  \forall \, u,w,v\in V. 
\end{align}
In addition, we assume that these conditions hold uniformly:
\begin{align}
m_{a}:=\underset{\mu \in \parset}{\inf}m_{a}(\mu)>0, \quad L_{a}:=\underset{\mu \in \parset}{\sup}L_{a}(\mu) < \infty.
\end{align}

For given $(g(\cdot;\mu),u_{o}) \in L^{2}(I;V') \times H$ we consider the quasilinear parabolic initial value problem of finding $u(t):=u(t;\mu) \in V, t \in I$ a.e. on the time interval $I=(0,T]$, such that 
\begin{align}\label{2.1.4}
  \dot{u}(t)+A(u(t),\mu)= g(t) \  \text{in} \ V', \ u(0)=u_{o} \ \text{in} \ H,
\end{align}
where $\dot{u}:=\frac{\partial u}{\partial t}$. We now define a space-time variational formulation of \eqref{2.1.4}. We use the trial space
\begin{align*}
\trialX:=W(0,T)=L^{2}(I;V)\cap H^{1}(I;V')=\{v \in L^{2}(I;V): v,\dot{v} \in  L^{2}(I;V') \}
\end{align*}
with the norm $\lVert w \rVert_{\trialX}^{2}:=\lVert \dot{w} \rVert_{L^{2}(I;V')}^{2}+\lVert w \rVert_{L^{2}(I;V)}^{2}$, and the test space  $\testY:=L^{2}(I;V) \times H$ with the norm $\lVert v \rVert_{\testY}^{2}:=\lVert v^{(1)} \rVert_{L^{2}(I;V)}^{2}+\lVert v^{(2)} \rVert_{H}^{2}$ for $v:=(v^{(1)},v^{(2)})$. The weak formulation of problem \eqref{2.1.4} reads: find $u:=u(\mu) \in \trialX$ such that
\begin{align}\label{2.1.5}
    B[u](u,v;\mu)=F(v;\mu), \quad \forall \, v\in \testY,
 \end{align}
where
 \begin{align}\label{2.1.6}
    B[u](u,v;\mu):=& \int_{I} \langle \dot{u},v^{(1)} \rangle_{V'V}+ a[u](u,v^{(1)};\mu)dt+\langle u(0),v^{(2)} \rangle_{H}, \ \text{and} \\
    F(v;\mu):=& \int_{I}\langle g(\mu),v^{(1)}\rangle_{V'V}dt+\langle u_{o},v^{(2)}\rangle_{H}.   
 \end{align}
Since $\trialX \hookrightarrow C(I;H)$, the initial value $u(0)$ is well-defined in $H$ \cite{zeidler2013linear}. We note that \eqref{2.1.2} implies coercivity of the quasilinear form $a[\cdot](\cdot,\cdot)$ and \eqref{2.1.3} implies hemicontinuity, i.e. the continuity of the mapping $s \rightarrow \langle A(u+sw),v \rangle_{V'V}$ for $s \in [0,1]$ and all $u,w,v \in V$. All together, the well-posedness of problem \eqref{2.1.5} follows, so that \eqref{2.1.5} admits a unique solution $u \in \trialX$, see e.g. \cite[Theorem 30.A]{zeidler2013nonlinear}.

\subsection{\textbf{Petrov-Galerkin Truth Approximation}}
\label{sec:2.2}
 From here onwards we omit the dependence on $\mu$ wherever appropriate. For the temporal discretization of \eqref{2.1.5} we use the time grid $0=t^{0}<t^{1}<...<t^{K}=T$ and set $I^{k}=(t^{k-1},t^{k}]$ for $k=1,...,K$. We set $\triangle t^{k}=t^{k}-t^{k-1}$ and define $\triangle t:=\max_{1\leq k \leq K} \triangle t^{k}$. For the spatial discretization we set $V_{h}=\text{span}\{\phi_{1},...,\phi_{\mathcal{N}_{h}}\}\subset V$, where $\dim {V_{h}}=\mathcal{N}_{h}$ and $h$ denotes the spatial discretization parameter. The functions $\phi_{i}$ will be defined in the numerical examples. With $\delta:=(\triangle t,h)$ we introduce the discrete trial space 
\begin{align*}
\trialXd:=\{u_{\delta} \in C^{0}(I;V),\restr{u_{\delta}}{I^{k}} \in \mathcal{P}_{1}(I^{k},V_{h}),\ k=1,...,K \} \subset \trialX 
\end{align*}
and the discrete test space 
\begin{align*}
\testYd:=\{v_{\delta}\in L^{2}(I;V),\restr{v_{\delta}}{I^{k}} \in \mathcal{P}_{0}(I^{k},V_{h}),\ k=1,...,K \}\times V_{h} \subset \testY.    
\end{align*}
With these choices of spaces the fully discrete truth approximation problem reads: find $u_{\delta}:=u_{\delta}(\mu) \in \trialXd$, such that $u^{0}_{\delta}:=u_{\delta}(0)=P_{H}^{h}u_{o}$ and
\begin{align}\label{2.2.0}
    B[u_{\delta}](u_{\delta},v_{\delta};\mu)=F(v_{\delta};\mu) \quad \forall v_{\delta} \in \testYd,
\end{align}
where $P_{H}^{h}:H \rightarrow V_{h}$ denotes the $H$-orthogonal projection onto $V_{h}$. It follows as for \eqref{2.1.5} that problem \eqref{2.2.0} admits a unique solution $u_{\delta} \in \trialXd$. The Petrov-Galerkin space-time discrete formulation \eqref{2.2.0} can be interpreted as Crank-Nicolson time-stepping scheme. Indeed, since the test space $\testYd$ consists of piecewise constant polynomials in time, the problem can be solved via the following procedure for $k=1,...,K$:  
\begin{align}\label{2.2.1}
     \int_{I^{k}}\langle \dot{u}_{\delta},v_{h} \rangle_{V'V}+a[u_{\delta}](u_{\delta},v_{h};\mu)dt=\int_{I^{k}}\langle g(\mu),v_{h} \rangle_{V'V}dt \quad \forall v_{h} \in V_{h}.
\end{align}
 Since the trial space $\trialXd$ consists of piecewise linear and continuous polynomials in time with the values $u_{\delta}^{k}:=u_{\delta}(t^{k})$ and $u_{\delta}^{k-1}:=u_{\delta}(t^{k-1})$, we can represent $u_{\delta}$ on $I^{k}$ as the linear function 
\begin{align}\label{2.2.2}
    u_{\delta}(t)=\frac{1}{\vartriangle t^{k}}\{(t^{k}-t)u_{\delta}^{k-1}+(t-t^{k-1})u_{\delta}^{k}\}, \ t\in I^{k}.
\end{align}
We use the representation \eqref{2.2.2} in \eqref{2.2.1}, test \eqref{2.2.1} against the basis functions $\phi_{i} \in V_{h} \ (i=1,...,\mathcal{N}_{h})$ and use the trapezoidal quadrature rule for the approximation of the appearing integrals. In this way we obtain the Crank-Nicolson time-stepping scheme, which for $k=1,...,K$ reads 
\begin{align}\label{2.2.3}
    (u_{\delta}^{k}-u_{\delta}^{k-1},\phi_{i})_{H}+\frac{\vartriangle t^{k}}{2} \{ a[u_{\delta}^{k}](u_{\delta}^{k},\phi_{i};\mu)+ a[u_{\delta}^{k-1}](u_{\delta}^{k-1},\phi_{i};\mu)\}=\\ \nonumber
    =\frac{\vartriangle t^{k}}{2} \{\langle g(t^{k};\mu),\phi_{i}\rangle_{V'V}+\langle g(t^{k-1};\mu),\phi_{i}\rangle_{V'V}\}, \quad 1\leq i \leq \mathcal{N}_{h}.
\end{align}
Here we recall that the initial condition $u_{\delta}^{0}$ is obtained as an $H$-orthogonal projection of $u_{o}$ onto $V_{h}$. Given the ansatz $u_{\delta}^{k}=\sum_{i=1}^{\mathcal{N}_{h}} u_{i}^{k} \phi_{i}$ and defining $\mathbf{u}_{\delta}^{k}:=\{u_{i}^{k}\}_{i=1}^{\mathcal{N}_{h}} \in \mathbb{R}^{\mathcal{N}_{h}}$, the resulting nonlinear algebraic equations are then solved by applying Newton's method for finding the root $\mathbf{u}_{\delta}^{k}$ of
\begin{align}\label{2.2.4}
\mathbf{G}_{h}(\mathbf{u}_{\delta}^{k};\mu):&=\frac{1}{\vartriangle t^{k}}\mathbf{M}_{h}(\mathbf{u}_{\delta}^{k}-\mathbf{u}_{\delta}^{k-1})-\frac{1}{2}[\mathbf{g}_{h}^{k}(\mu)+\mathbf{g}_{h}^{k-1}(\mu)] \\ \nonumber
&+\frac{1}{2}[\mathbf{A}_{h}(\mathbf{u}_{\delta}^{k};\mu)\mathbf{u}_{\delta}^{k}+\mathbf{A}_{h}(\mathbf{u}_{\delta}^{k-1};\mu)\mathbf{u}_{\delta}^{k-1}],
\end{align}
where $\mathbf{M}_{h}:=\{\langle \phi_{i},\phi_{j} \rangle_{H}\}_{i,j=1}^{\mathcal{N}_{h}},  \mathbf{A}_{h}(\mathbf{u}_{\delta}^{k};\mu):=\{a[u_{\delta}^{k}](\phi_{i},\phi_{j};\mu)\}_{i,j=1}^{\mathcal{N}_{h}}\in \mathbb{R}^{\mathcal{N}_{h} \times \mathcal{N}_{h}}$ and $\mathbf{g}_{h}^{k}(\mu):=\{\langle g(t^{k};\mu),\phi_{i}\rangle_{V'V}\}_{i=1}^{\mathcal{N}_{h}}\in \mathbb{R}^{\mathcal{N}_{h}}$. The initial condition for \eqref{2.2.4} is given by $\mathbf{u}_{\delta}^{0}:=\{(u_{o},\phi_{i})_{H}\}_{i=1}^{\mathcal{N}_{h}}\in \mathbb{R}^{\mathcal{N}_{h}}$. The strong monotonicity of the quasilinear form $\eqref{2.1.2}$ guarantees that the equation \eqref{2.2.4} admits a unique root $\mathbf{u}_{\delta}^{k}$ for every parameter $\mu \in \parset$.

The Newton's iteration for finding a root of \eqref{2.2.4} reads: starting with $\mathbf{u}_{\delta}^{k,(0)}$, for $z=0,1,...$ solve the linear system
\begin{align}\label{2.2.4b}
    \mathbf{J}_{h}(\mathbf{u}_{\delta}^{k,(z)};\mu)\delta \mathbf{u}_{\delta}^{k,(z)}=-\mathbf{G}_{h}(\mathbf{u}_{\delta}^{k,(z)};\mu)
\end{align}
to obtain $\delta \mathbf{u}_{\delta}^{k,(z)}$, and then update the solution $\mathbf{u}_{\delta}^{k,(z+1)}:=\mathbf{u}_{\delta}^{k,(z)}+\delta \mathbf{u}_{\delta}^{k,(z)}$. The system Jacobian matrix is given by
\begin{align}\label{2.2.5}
    \mathbf{J}_{h}(\mathbf{u}_{\delta}^{k};\mu)=\frac{1}{\vartriangle t^{k}}\mathbf{M}_{h}+\frac{1}{2}\mathbf{A}'_{h}(\mathbf{u}_{\delta}^{k};\mu),
\end{align}
where $\mathbf{A}'_{h}(\mathbf{u}_{\delta}^{k};\mu):=\{da[u_{\delta}^{k}](\phi_{i},\phi_{j};\mu)\}_{i,j=1}^{\mathcal{N}_{h}}\in \mathbb{R}^{\mathcal{N}_{h} \times \mathcal{N}_{h}}$. Here we assume the existence of the Fr\'echet derivative $A'(u;\mu) :V \times \parset \rightarrow V'$ of the nonlinear operator $A(u;\mu)$ for every parameter $\mu \in \parset$, which induces the corresponding bilinear form $\langle A'(u;\mu)v,w\rangle_{V'V}=da[u](v,w;\mu)$. We will specify it later for our examples. We note that $\mathbf{A}'_{h}(\mathbf{u}_{\delta}^{k};\mu)$ is positive definite, since $da[u](\cdot,\cdot;\mu)$ is coercive due to the strong monotonicity of $A$; therefore, the system \eqref{2.2.4b} admits a unique solution.

\section{The Reduced Basis method}
\label{sec:3}
In this section we introduce the reduced basis model and its numerical realization. Then we introduce our a-posteriori error bound and discuss its efficient evaluation. 

\subsection{\textbf{Empirical interpolation of the nonlinearity}}
\label{sec:3.1}

We use the Empirical Interpolation Method (EIM) \cite{barrault2004empirical} to ensure the availability of an affine decomposition for the quasilinear form $a[u_{\delta}^{k}](\cdot,\cdot;\mu)$ for every parameter  $\mu \in \parset$. We then need to find a parameter-separable (affine) counterpart $\nu_{M}(\cdot;\mu)$ of the nonlinear non-affine function $\nu(\cdot;\mu)$. For EIM nonlinearity approximation, we treat time as an additional parameter in the problem, thus we set $\mathbb{I}:=\{1,..., K\}$ as our discrete time set. We construct with the Algorithm \ref{alg:EIM algorithm} the nested sample sets $S_{M}^{\nu} \subset \parset$ and $\mathbb{I}_{M}^{\nu}\subset \mathbb{I}$, where $S_{M}^{\nu}:=\{ \mu_{1}^{\nu} \in \parset,...,\mu_{M}^{\nu} \in \parset \}$ and $\mathbb{I}_{M}^{\nu}:=\{ k_{1}^{M} \in \mathbb{I},...,k_{M}^{M} \in \mathbb{I} \}$, and associated approximation spaces $W_{M}^{\nu}:=\text{span}\{\nu(u_{\delta}^{k_{m}^{M}}(\cdot;\mu_{m}^{\nu});\mu_{m}^{\nu}),\ 1\leq m \leq M \}=\text{span}\{q_{1},...,q_{M}\}$. Algorithm \ref{alg:EIM algorithm} also provides the nested sets of interpolation points $T_{M}=\{x_{1}^{M},...,x_{M}^{M}\}, \ 1\leq M \leq M_{\text{max}}$. We build an affine approximation $\nu_{M}(u_{\delta}^{k}(x);\mu)$ of $\nu(u_{\delta}^{k}(x);\mu)$ for our time-marching scheme according to
\begin{align}\label{3.1.2}
\nu(u_{\delta}^{k}(x);\mu)\approx& \sum_{m=1}^{M}\varphi_{m}^{k}(\mu)q_{m}(x) \\ \nonumber
=& \sum_{m=1}^{M}(\mathbf{B}_{M}^{-1}\nu_{\mu}^{k})_{m}q_{m}(\hat{x}):= \nu_{M}(u_{\delta}^{k}(x);\mu),
\end{align}
where $\nu_{\mu}^{k}:=\{\nu(u_{\delta}^{k}(x_{m}^{M});\mu) \}_{m=1}^{M} \in \mathbb{R}^{M}$ and $\mathbf{B}_{M}\in \mathbb{R}^{M \times M}$ is the lower triangular interpolation matrix $(\mathbf{B}_{M})_{ij}=q_{j}(x_{i})$ with $(\mathbf{B}_{M})_{ii}=1 \ (i=1,...,M)$ by construction.  

\begin{algorithm}[H]
  \caption{: EIM algorithm}
  \label{alg:EIM algorithm}
  \hspace*{\algorithmicindent} \textbf{Input:} Stopping tolerance $\epsilon_{EIM}$, max. number of iterations $M_{\text{max}}$, parameter set $\parset$. \\
  \hspace*{\algorithmicindent} \textbf{Output:} Nested approximation spaces $\{W_{m}^{\nu}\}_{m=1}^{M}$,  nested interpolation points $\{T_{m}\}_{m=1}^{M}$. \\
  \begin{algorithmic}[1]
  \STATE{$ (\mu_{1}^{\nu},k_{1}^{M}) := \underset{(\mu,k) \in \parset \times \mathbb{I}}{\mathrm{arg}\max} \ \lVert \nu(u_{\delta}^{k}(\cdot);\mu)\rVert_{L^{\infty}(\Omega)} $}
  \STATE{$S_{1}^{\nu}\times \mathbb{I}_{1}^{\nu}:=  \{\mu^{\nu}_{1}\}\times \{k_{1}^{M}\}$}
  \STATE{$r_{1}(x):=\nu(u_{\delta}^{k_{m}^{M}}(x;\mu^{\nu}_{m});\mu_{m}^{\nu})$}
  \STATE{$x_{1}^{M}:=\underset{x \in \Omega}{\mathrm{arg}\max}|r_{1}(x)|, \quad q_{m}:=r_{1}/r_{1}(x_{1}^{M})$}
  \STATE{$T_{1} := \{x_{m}^{M}\}, \quad Q_{1}:=\{q_{1}\},\quad W_{1}^{\nu}:= \text{span}(Q_{1})$}
  \WHILE{$2\leq m \leq M_{\text{max}}$ and $\delta_{m}^{max}>\epsilon_{EIM}$} 
  \STATE{$ (\mu_{m}^{\nu},k_{m}^{M}) := \underset{(\mu,k) \in \parset \times \mathbb{I}}{\mathrm{arg}\max} \ \lVert \nu(u_{\delta}^{k}(\cdot);\mu)- \nu_{m}(u_{\delta}^{k}(\cdot);\mu) \rVert_{L^{\infty}(\Omega)} $}
  \STATE{$ \delta_{m}^{max} := \underset{(\mu,k) \in \parset \times \mathbb{I}}\max\ \lVert \nu(u_{\delta}^{k}(\cdot);\mu)- \nu_{m}(u_{\delta}^{k}(\cdot);\mu) \rVert_{L^{\infty}(\Omega)} $}
  \STATE{$S_{m}^{\nu}:= S_{m-1}^{\nu}\cup \{\mu^{\nu}_{m}\}, \quad \mathbb{I}_{m}^{\nu}:=\mathbb{I}_{m-1}^{\nu}\cup \{k_{m}^{M}\}$}
  \STATE{$r_{m}(x):=\nu(u_{\delta}^{k_{m}^{M}}(x;\mu^{\nu}_{m});\mu_{m}^{\nu})-\nu_{m}(u_{\delta}^{k_{m}^{M}}(x;\mu^{\nu}_{m});\mu_{m}^{\nu})$}
  \STATE{$x_{m}^{M}:=\underset{x \in \Omega}{\mathrm{arg}\max}|r_{m}(x)|, \quad q_{m}:=r_{m}/r_{m}(x_{m}^{M})$}
   \STATE{$T_{m} := T_{m-1}\cup \{x_{m}^{M}\}, \quad Q_{m}:=Q_{m-1}\cup\{q_{m}\}, \quad W_{m}^{\nu}:= \text{span}(Q_{m})$}
   \STATE{$m \leftarrow m+1$}
  \ENDWHILE
  \end{algorithmic}
\end{algorithm}

We then have the EIM approximation $\tilde{a}[\cdot](\cdot,\cdot;\mu)$ of the quasilinear form $a[\cdot](\cdot,\cdot;\mu)$, which admits the affine decomposition
 \begin{align}\label{3.1.3}
\tilde{a}[u_{\delta}^{k}](u_{\delta}^{k},v;\mu)=\sum_{m=1}^{M}\varphi_{m}^{k}(\mu) \tilde{a}_{m}(u_{\delta}^{k},v), \quad \tilde{a}_{m}(u_{\delta}^{k},v)=\int_{\Omega}q_{m} \nabla u_{\delta}^{k} \cdot \nabla v \ dx.
 \end{align}
We also assume the affine decomposition 
 \begin{align}\label{3.1.3b}
     \langle g(t^{k};\mu),v \rangle_{V'V}= \sum_{q=1}^{Q_{g}}\theta_{g,q}^{k}(\mu)\langle g_{q},v \rangle_{V'V}
 \end{align}
for the right-hand side, where $\theta_{g,q}^{k} :\parset \rightarrow \mathbb{R}$ are parameter-dependent functions and parameter-independent forms  $g_{q} : V \rightarrow \mathbb{R}$, $k=1,...,K$, $q=1,...,Q_{g}$. If \eqref{3.1.3b} is not available, the EIM procedure can be similarly applied.

\subsection{\textbf{Reduced basis approximation with the POD-Greedy method}}
\label{sec:3.2}

The idea of the reduced-basis approximation consists in replacing the ``truth" (high-dimensional) space $V_{h}$ in the definition of $\trialXd$ and $\testYd$ by a low-dimensional subspace $V_{N} \subset V_{h}$. With $V_{N}$ available we introduce the corresponding reduced trial space
\begin{align*}
\mathcal{X}_{\triangle t,N}:=\{u_{N}\in C^{0}(I;V),\restr{u_{N}}{I^{k}} \in \mathcal{P}_{1}(I^{k},V_{N}),\ k=1,...,K \}
\end{align*}
and the reduced test space
\begin{align*}
\mathcal{Y}_{\triangle t,N}:=\{v_{N}\in L^{2}(I;V),\restr{v_{N}}{I^{k}} \in \mathcal{P}_{0}(I^{k},V_{N}), \ k=1,...,K \} \times V_{N}.   
\end{align*}
We construct $V_{N}:=\text{span}\{\xi_{1},...,\xi_{N}\} \subset V_{h}$ by the POD-Greedy procedure in Algorithm \ref{alg:RB-Greedy algorithm}, compare e.g. \cite{haasdonk2008reduced}. In our setting, the POD-Greedy alogorithm constructs iteratively nested spaces $V_{n},\ 1\leq n \leq N$ using an a-posteriori error estimator $\bigtriangleup(Y;\mu)$ (see the next section for details on a-posteriori error analysis), which predicts the expected approximation error for a given parameter $\mu$ in the space $Y:=\mathcal{Y}_{\triangle t,n}$. We want the expected approximation error to be less than the prescribed tolerance $\varepsilon_{RB}$.  We initiate the algorithm with the choice of the initial basis vector  $\xi_{1}:=u_{\delta}^{0} / \lVert u_{\delta}^{0} \rVert_{V}$; this choice is motivated by the assumption in Proposition 1. The snapshots  $u_{\delta}(\mu)$ for the procedure are provided by the parametrized  ``truth" approximation \eqref{2.2.0}. Next we proceed as stated in the following Algorithm \ref{alg:RB-Greedy algorithm}.
\begin{algorithm}[H]
  \caption{: POD-Greedy algorithm}
  \label{alg:RB-Greedy algorithm}
  \hspace*{\algorithmicindent} \textbf{Input:} Tolerance $\varepsilon_{RB}$, max. number of iterations $N_{\text{max}}$, $V_{1}=\text{span}\{\xi_{1}\}$, parameter set $\parset$. \\
  \hspace*{\algorithmicindent} \textbf{Output:} RB spatial spaces $\{V_{n}\}_{n=1}^{N}$, RB trial spaces $\{\mathcal{X}_{\triangle t,n}\}_{n=1}^{N}$, RB test spaces $\{\mathcal{Y}_{\triangle t,n}\}_{n=1}^{N}$.
  \begin{algorithmic}[1]
   \WHILE { $2 \leq n \leq N_{\text{max}}$ and $\varepsilon_{n} :=\underset{\mu \in \parset_{train}}{\max} \bigtriangleup(\mathcal{Y}_{\triangle t,n},\mu)> \varepsilon_{RB}$}
  \STATE {$[\varepsilon_{n}, \mu_{n}] \leftarrow \underset{\mu \in \parset_{train}}{\mathrm{arg}\max}\bigtriangleup(\mathcal{Y}_{\triangle t,n-1},\mu)$}
  \STATE $e_{n}^{k}:=u_{\delta}^{k}(\mu_{n})-P_{V} u_{\delta}^{k}(\mu_{n}), \ k=1,...,K$
  \STATE $\xi_{n}:=\text{POD}_{1}(\{e_{n}^{k}\}_{k=1}^{K})$
  \STATE{$V_{n}:=V_{n-1} \bigoplus \text{span}\{\xi_{n}\}$}
  \STATE $\mathcal{X}_{\triangle t,n} \leftarrow \mathcal{X}_{\triangle t,n-1}, \quad \mathcal{Y}_{\triangle t,n} \leftarrow \mathcal{Y}_{\triangle t,n-1}$
  \STATE{$n \leftarrow n+1$}
  \ENDWHILE
  \end{algorithmic}
\end{algorithm}
\noindent
In Algorithm \ref{alg:RB-Greedy algorithm}, $P_{V}: V_{h} \rightarrow V_{n}$ denotes the $V$-orthogonal projection, and the operation $\text{POD}_{1}(\{e_{n}^{k}\}_{k=1}^{K})$ denotes the extraction of the dominant mode of  the Proper Orthogonal Decomposition (see, e.g. \cite{volkwein2013proper}). We also note that more modes can be extracted in every step of the algorithm: it reduces the offline computational time, but there is no guarantee that the produced basis will be of the smallest possible dimension.

The reduced-basis approximation of problem \eqref{2.2.0} reads: find $u_{N}:=u_{N}(\mu) \in \mathcal{X}_{\triangle t,N}$, such that  $u^{0}_{N}:=u_{N}(0)=P_{H}^{N}u_{o}$ and
\begin{align}\label{3.1.4}
    \tilde{B}[u_{N}](u_{N},v_{N};\mu)=\tilde{F}(v_{N};\mu) \quad \forall v_{N} \in \mathcal{Y}_{\triangle t,N},
\end{align}
where
\begin{align*}
     \tilde{B}[u_{N}](u_{N},v_{N};\mu)&=\int_{I} \langle \dot{u}_{N},v_{N}^{(1)} \rangle_{V'V}+ \tilde{a}[u_{N}](u_{N},v_{N}^{(1)};\mu)dt
     +\langle P_{H}^{N}u_{o},v_{N}^{(2)} \rangle_{H},\\
 \tilde{F}(v_{N};\mu):&= \int_{I}\langle g(\mu),v_{N}^{(1)}\rangle_{V'V}dt+\langle u_{\delta}^{0},v_{N}^{(2)}\rangle_{H},     
\end{align*}
 and $P_{H}^{N}: V_{h}\rightarrow V_{N}$ denotes the $H$-orthogonal projection onto $V_{N}$. For mathematical convenience, we assume that the EIM approximation $\tilde{a}[\cdot](\cdot,\cdot;\mu)$ is sufficiently accurate in the sense that the form $\tilde{a}[\cdot](\cdot,\cdot;\mu)$ is strongly monotone with monotonicity constant $\tilde{m}_{a}(\mu):=m_{a}(\mu)-\epsilon_{a}>0$, where $\epsilon_{a} \in \mathbb{R}_{+}$ is small enough and is related to the EIM approximation error. Then it follows as for \eqref{2.1.5} that the problem $\eqref{3.1.4}$ admits a unique solution $u_{N}(\mu) \in \mathcal{X}_{\triangle t,N}$ for all $\mu \in \parset$. However, in the EIM practice it is difficult to check this property a-priori, so that arguing the well-posedness of the upcoming discrete systems \eqref{discrete_root} and \eqref{rb_newton} in general is not possible.
 
 The problem \eqref{3.1.4} can be interpreted as the reduced-basis approximation of the Crank-Nicolson time-marching scheme with the EIM approximation of the nonlinearity, i.e.
\begin{align}\label{RB_scheme}
     (u_{N}^{k}-u_{N}^{k-1},v_{N}^{(1)})_{H}&+\frac{\vartriangle t^{k}}{2}\{ \tilde{a}[u_{N}^{k}](u_{N}^{k},v_{N}^{(1)};\mu)+\tilde{a}[u_{N}^{k-1}](u_{N}^{k-1},v_{N}^{(1)};\mu)\}\\ \nonumber
    &=\frac{\vartriangle t^{k}}{2} \{\langle g(t^{k};\mu),v_{N}^{(1)}\rangle_{V'V}+\langle g(t^{k-1};\mu),v_{N}^{(1)}\rangle_{V'V} \}, 
\end{align}
where the initial condition $u_{N}^{0}$ is obtained as an $H$-projection of $u_{\delta}^{0}$ onto $V_{N}$.  The resulting nonlinear algebraic equations are then solved with the RB counterpart of Newton's method by finding the root of
 \begin{align}\label{discrete_root}
\mathbf{G}_{N,M}(\mathbf{u}_{N}^{k};\mu)&=\frac{1}{\vartriangle t^{k}}\mathbf{M}_{N}(\mathbf{u}_{N}^{k}-\mathbf{u}_{N}^{k-1})-\frac{1}{2}[\mathbf{g}_{N}^{k}(\mu)+\mathbf{g}_{N}^{k-1}(\mu)] \\\nonumber
&+\frac{1}{2}[\mathbf{A}_{N,M}(\mu)\mathbf{u}_{N}^{k}+\mathbf{A}_{N,M}(\mu)\mathbf{u}_{N}^{k-1}],
 \end{align}
 where $\mathbf{M}_{N}:=\{\langle \xi_{i},\xi_{j} \rangle_{H}\}_{i,j=1}^{N}, \mathbf{A}_{N,M}(\mu):=\{\tilde{a}[u_{N}^{k}](\xi_{i},\xi_{j};\mu)\}_{i,j=1}^{N} \in \mathbb{R}^{N \times N}$ and $\mathbf{g}_{N}^{k}(\mu):=\{\langle g(t^{k};\mu),\xi_{i}\rangle_{V'V}\}_{i=1}^{N}\in \mathbb{R}^{N}$. The initial condition is given by $\mathbf{u}_{N}^{0}:=\{(u_{\delta}^{0},\xi_{i})_{H}\}_{i=1}^{N} \in \mathbb{R}^{N}$. The strong monotonicity of the quasilinear form $\eqref{3.1.3}$ guarantees that the equation \eqref{discrete_root} admits a unique root $\mathbf{u}_{N}^{k}$ for every parameter $\mu \in \parset$.
 
 The Newton's iteration for finding a root of \eqref{discrete_root} reads: starting with $\mathbf{u}_{N}^{k,(0)}$, for $z=0,1,...$ solve the linear system
\begin{align}\label{rb_newton}
    \mathbf{J}_{N,M}(\mathbf{u}_{N}^{k,(z)};\mu)\delta \mathbf{u}_{N}^{k,(z)}=-\mathbf{G}_{N,M}(\mathbf{u}_{N}^{k,(z)};\mu)
\end{align}
to obtain $\delta \mathbf{u}_{N}^{k,(z)}$, and then update the solution $\mathbf{u}_{N}^{k,(z+1)}:=\mathbf{u}_{N}^{k,(z)}+\delta \mathbf{u}_{N}^{k,(z)}$. The system Jacobian matrix is given by
\begin{align}\label{rb_newton_matrix}
    \mathbf{J}_{N,M}(\mathbf{u}_{N}^{k};\mu)=\frac{1}{\vartriangle t^{k}}\mathbf{M}_{N}+\frac{1}{2}\mathbf{A}'_{N,M}(\mathbf{u}_{N}^{k};\mu).
\end{align}
If the mapping $\mu \mapsto \mathbf{A}'_{N,M}(\cdot;\mu)$ is bounded in $\mu \in \parset$, then for $\vartriangle t^{k} \leq C(\parset)$, where $C(\parset)>0$ is some constant, the Jacobian matrix \eqref{rb_newton_matrix} is invertible. We will comment on the computation of the reduced parametrised counterpart $\mathbf{A}'_{N,M}(\mu):=\{d\tilde{a}[u_{N}^{k}](\xi_{i},\xi_{j};\mu)\}_{i,j=1}^{N} \in \mathbb{R}^{N \times N}$ of $\mathbf{A}'_{h}(\mathbf{u}_{\delta}^{k})$ in \eqref{rb_newton_matrix}. We have
\begin{align}\label{3.1.5}
    \tilde{a}[u_{N}^{k}](u_{N},\xi_{i};\mu)=\sum_{j=1}^{N}\sum_{m=1}^{M}\varphi_{m}^{k}(\mu)\tilde{a}_{m}(\xi_{j},\xi_{i})u_{N,j}^{k}, \quad 1\leq i \leq N.
\end{align}
With the EIM approximation of the nonlinearity it follows that
\begin{align}\label{3.1.6}
\sum_{s=1}^{M}(\mathbf{B}_{M})_{m,s}\varphi_{m,s}^{k}(\mu)&=\nu(u_{N}^{k}(x_{m}^{M};\mu);\mu), \quad 1\leq m\leq M \\ \nonumber
    &=\nu(\sum_{n=1}^{N}u_{N,n}^{k}\xi_{n}(x_{m}^{M});\mu),\quad 1\leq m\leq M.
\end{align}
Plugging \eqref{3.1.6} into \eqref{3.1.5} results in
\begin{align}\label{3.1.7}
   \tilde{a}[u_{N}^{k}](u_{N},\xi_{i};\mu)=\sum_{j=1}^{N}\sum_{m=1}^{M}\mathbf{D}_{i, m}^{N,M}(\mu)\nu(\sum_{n=1}^{N}u_{N,n}^{k}\xi_{n}(x_{m}^{M});\mu)u_{N,j}^{k} 
\end{align}
with $\mathbf{D}^{N,M}(\mu)=\mathbf{A}_{N,M}(\mu) (\mathbf{B}_{M})^{-1} \in \mathbb{R}^{N \times M}$. Taking the derivative of \eqref{3.1.7} with respect to the components $u_{N,j}^{k}(\mu), \ 1\leq j \leq N$ , we derive the formula for  $\mathbf{A}'_{N,M}(\mu)=\mathbf{A}_{N,M}(\mu)+\mathbf{E}_{N,M}(\mu)$, where 
\begin{align}\label{3.1.8}
   (\mathbf{E}_{N,M})_{i,j}=\sum_{s=1}^{N}u_{N,s}^{k} \sum_{m=1}^{M}\mathbf{D}_{i, m}^{N,M}(\mu)\partial_{1}\nu(u_{N}^{k}(x_{m}^{M});\mu)
\end{align}
We will give the exact form of $\partial_{1}\nu(u_{N}^{k}(x_{m}^{M});\mu)$ in the upcoming examples. 

The proposed reduced numerical scheme contains parameter-separable matrices and thus allows offline-online decomposition. The offline phase (model construction) depends on expensive high-dimensional finite element simulations and thus on $\mathcal{N}$, but should be performed only once. However, the assembling of all the high-dimensional parameter-dependent quantities is computationally simplified due to the affine dependence on the parameters \eqref{3.1.3},\eqref{3.1.3b}. In the online phase (RB model simulation) the computational complexity scales polynomially in $N$ and $M$, independently of $\mathcal{N}$ and thus is inexpensive. The operation count associated with each Newton update of the residual $\mathbf{G}_{N,M}(\mathbf{u}_{N}^{k,(z)})$ in the online phase is $\mathcal{O}(N^{2}Q_{a}+N^{2}+M^{2}+NQ_{f_{o}})$ and the Jacobian $\mathbf{J}_{N,M}(\mathbf{u}_{N,M}^{k,(z)})$ is assembled at cost $\mathcal{O}(MN^{3})$ with the dominant cost of assembling $\mathbf{E}_{N,M}(\mu)$, and then inverted at cost $\mathcal{O}(N^{3})$.

\subsection{\textbf{Reduced basis certification}}
\label{sec:3.3}
An important ingredient of the reduced basis methodology is the verification of the error (certification of the reduced basis method). In the present work we provide an a-posteriori error bound, based on the residual, which allows quick evaluation. We denote by $R(\cdot;\mu) \in \testY'$  the residual of the problem, defined naturally as: 
\begin{align}\label{3.2.1}
   R(v;\mu):=F(v;\mu)-\tilde{B}[u_{N}](u_{N},v;\mu)=\int_{I}\langle r(t;\mu),v \rangle_{V'V}dt  \quad \forall v \in \testYd.
\end{align}
We have the following
\begin{prop}[A-posteriori Error Bound]
 Let $m_{a}(\mu)>0$ be a monotonicity constant from \eqref{2.1.2} and assume that $u^{0}_{\delta} \in V_{N}$. Then the error $e(\mu)=u_{\delta}(\mu)-u_{N}(\mu)$ of the reduced basis approximation is bounded by
 \begin{align}\label{3.2.2}
\lVert e(\mu)\rVert_{\testY}\leq \frac{1}{m_{a}(\mu)}(\lVert R(\cdot;\mu)\rVert_{\testY'}+\delta_{M}(\mu)\lVert u_{N}(\mu) \rVert_{L^{2}(I;V)})=:\bigtriangleup_{N,M}(\mu),
 \end{align}
where
 \begin{align}\label{3.2.3}
\delta_{M}(\mu)=\sup_{t \in I}\sup_{x \in \Omega}| \nu_{M}(u_{N}(x,t);\mu) - \nu(u_{N}(x,t);\mu)|
\end{align}
denotes the approximation error of the nonlinearity.
\end{prop}

\begin{proof}
Since in the case $e=0$ there is nothing to show, we assume that $e \neq 0$. We have $u_{\delta}^{0} \in V_{N}$ and $\restr{P_{H}^{N}}{V_{N}}=Id$, therefore $u_{N}^{0}:=P_{H}^{N}u_{\delta}^{0}=u_{\delta}^{0}$. It implies that $\lVert e(0)\rVert_{H}=0$, $\lVert e \rVert_{\testY}=\lVert e \rVert_{L^{2}(I;V)}$ and $\lVert R(\cdot;\mu)\rVert_{\testY'}=\lVert R(\cdot;\mu)\rVert_{L^{2}(I;V')}$. First we obtain the following estimate by applying Cauchy-Schwartz inequality: 
\begin{align}\label{first_ineq}
\tilde{a}[u_{N}](u_{N},e;\mu)-a[u_{N}](u_{N},e;\mu)=\int_{\Omega}[\nu_{M}(u_{N};\mu)-\nu(u_{N};\mu)]\nabla u_{N}\cdot \nabla e \ dx \\ \nonumber
\leq \sup_{x \in \Omega}|\nu_{M}(u_{N}(x,\cdot);\mu) - \nu(u_{N}(x,\cdot);\mu)| \ \lVert u_{N} \rVert_{V}\lVert e \rVert_{V}.
\end{align}
Integrating \eqref{first_ineq} in $t$ and  applying the Cauchy-Schwartz inequality to the corresponding integral we get:
\begin{align*}
\int_{I} \tilde{a}[u_{N}](u_{N},e;\mu)-a[u_{N}](u_{N},e;\mu)dt \leq \delta_{M}(\mu) \lVert u_{N} \rVert_{L^{2}(I;V)}\lVert e\rVert_{\testY}.
\end{align*}
We then use the identity
\begin{align}\label{identity}
\int_{I} \langle \dot{e},e \rangle_{V'V}dt=  \frac{1}{2}\lVert e(T) \lVert_{H}^{2}-\frac{1}{2}\lVert e(0)\lVert_{H}^{2} 
\end{align}
together with the strong monotonicity condition \eqref{2.1.2} and the estimate above to derive the bound:
\begin{align*}
 m_{a}(\mu) \lVert e \rVert_{\testY}^{2}\leq & \int_{I}  a[u_{\delta}](u_{\delta},e;\mu)-a[u_{N}](u_{N},e;\mu) dt+\frac{1}{2}\lVert e(T) \rVert_{H}^{2}\\
= \int_{I} \langle \dot{e},e \rangle_{V'V} dt+& \int_{I} a[u_{\delta}](u_{\delta},e;\mu)-a[u_{N}](u_{N},e;\mu) dt+\frac{1}{2}\lVert e(0) \rVert_{H}^{2}  \\
= \int_{I} \langle \dot{e},e \rangle_{V'V}dt+&\int_{I} a[u_{\delta}](u_{\delta},e;\mu)-\tilde{a}[u_{N}](u_{N},e;\mu)dt+\lVert e(0)\rVert_{H}^{2}\\
+ &\int_{I} \tilde{a}[u_{N}](u_{N},e;\mu)-a[u_{N}](u_{N},e;\mu)dt  \\
&\leq  \lVert R(\cdot;\mu) \rVert_{\testY'} \lVert e \rVert_{\testY}+\delta_{M}(\mu) \lVert u_{N} \rVert_{L^{2}(I;V)} \lVert e \rVert_{\testY},   
\end{align*}
where we added and subtracted $\tilde{a}[u_{N}](u_{N},e;\mu)$ to get the definition of the residual \eqref{3.2.1}. Dividing both sides by $ \lVert e \rVert_{\testY} $ yields the result. 
\end{proof}

We note that the assumption $u_{\delta}^{0} \in V_{N}$ implies that $\lVert e(0)\rVert_{H}=0$. We can guarantee this by choosing $\xi_{1}:=u_{\delta}^{0} / \lVert u_{\delta}^{0} \rVert_{V}$ as the initial basis for $V_{N}$ in the POD-Greedy procedure.
 
The dual norm of the residual $\rVert R(\cdot;\mu)\lVert_{\testY'}$ in \eqref{3.2.2} is not available analytically, but by using $\testYd$ as the underlying test space, it becomes a computable quantity. The computation of $\rVert R(\cdot;\mu)\lVert_{\testY'}$ requires the knowledge of its Riesz representer $v_{\delta,R}(\mu) \in \testYd $. Thanks to the Riesz representation theorem, on the discrete level it can be obtained from the equation
\begin{align}\label{3.2.4}
(v_{\delta,R}(\mu),v_{\delta})_{\testY}=R(v_{\delta};\mu) \quad \forall v_{\delta}\in \testYd.
\end{align}
Since the test space $\testYd$ consists of piecewise constant polynomials  in time, the problem \eqref{3.2.4} can be solved via the time-marching procedure for $k=1,...,K$ as follows:  
\begin{align}\label{3.2.5}
 \int_{I^{k}}\langle v_{\delta,R}(t;\mu),v_{h} \rangle_{V}dt=\int_{I^{k}}\langle r(t;\mu),v \rangle_{V'V}dt \quad \forall v_{h} \in V_{h}.
 \end{align}
We note that $v_{R}^{k}(\mu):=\restr{v_{\delta,R}(\mu)}{I^{k}}$ is constant in time, hence the integration on the left-hand side of \eqref{3.2.5} is exact. For the right-hand side of \eqref{3.2.5} we represent $u_{N}(\mu) \in \mathcal{X}_{\triangle t, N}$ as the linear function  \eqref{2.2.2} on $I^{k}$ and use it as an input for the residual \eqref{3.2.1}. We then apply the trapezoidal quadrature rule for the approximate evaluation of the integral. The quadrature rule is chosen such that the quadrature error is of the size of the error of the truth Crank-Nicolson solution. We thus need to solve the following problems:
\begin{align}\label{3.2.6}
 \langle v_{R}^{k}(\mu),v_{h} \rangle_{V}= R^{k}(v_{h};\mu) \quad \forall v_{h} \in V_{h} \ (k=1,...,K),
\end{align}
where the right-hand side is given by
\begin{align}\label{3.2.7}
    R^{k}(v_{h};\mu)=\frac{1}{2}[\langle g(t^{k};\mu)+g(t^{k-1};\mu),v_{h}\rangle_{V'V}-\tilde{a}[u_{N}^{k}](u_{N}^{k},v_{h};\mu) \\ \nonumber
    -\tilde{a}[u_{N}^{k-1}](u_{N}^{k-1},v_{h};\mu)]  - \frac{1}{\triangle t^{k}}\langle u_{N}^{k}-u_{N}^{k-1},v_{h} \rangle_{H}. 
\end{align}
Therefore the computation of the Riesz representer leads to a sequence of $K$ uncoupled spatial problems in $V_{h}$. The parameter separability structure of the residual
\begin{align*}
R^{k}(v_{h};\mu)=\sum_{q=1}^{Q_{R}}\theta_{R,q}^{k}(\mu)R_{q}(v_{h})
\end{align*}
is transferred by the linearity of the Riesz isomorphism to the parameter separability of its Riesz representer $v^{k}_{R}(\mu)$ together with the parameter dependent functions $\theta_{R,q}^{k}: \parset \rightarrow \mathbb{R}$. Therefore, for $1\leq q \leq Q_{R}$ we have
\begin{align}\label{3.2.8}
v^{k}_{R}(\mu)=\sum_{q=1}^{Q_{R}}\theta_{R,q}^{k}(\mu)v_{R,q} \ \text{with} \ (v_{R,q},v_{h})_{V}=R_{q}(v_{h}) \quad \forall v_{h} \in V_{h}.
\end{align}
Finally we state the formulas for the residual norm  as well as the spatio-temporal norm of $u_{N}$. Since $\restr{v_{\delta,R}(\mu)}{I^{k}}$ is constant in time, the integration on $I^{k}$ is exact and we can compute the spatio-temporal norm of $v_{\delta,R}(\mu)$ as follows:
\begin{align*}
\lVert v_{\delta,R}(\mu) \rVert_{\testY}^{2}=\sum_{k=1}^{K} \triangle t^{k}\lVert v_{R}^{k}(\mu) \rVert_{V}^{2}=\sum_{k=1}^{K} \triangle t^{k}\Theta^{k}_{R}(\mu)^{T}\mathbf{G}_{R}\Theta^{k}_{R}(\mu),
\end{align*}
where $\mathbf{G}_{R}:=[\langle v_{R,q},v_{R,q'}\rangle]_{q,q'=1}^{Q_{R}}\in \mathbb{R}^{Q_{R}\times Q_{R}}$ and ${\Theta}^{k}_{R}(\mu):=[\theta_{R,q}^{k}(\mu)]_{q=1}^{Q_{R}}\in \mathbb{R}^{Q_{R}}$. The isometry of the Riesz isomorphism implies that $\lVert R(\cdot;\mu) \rVert_{\testY'}=\lVert v_{\delta,R}(\mu) \rVert_{\testY}$. Since $\restr{u_{N}(\mu)}{I^{k}}$ is a linear function in time, the trapezoidal quadrature rule on $I^{k}$ is exact. We then can compute the spatio-temporal norm $\lVert u_{N} \rVert_{\testY}$ of $u_{N} \in  \mathcal{X}_{\triangle t, N}$ according to
\begin{align*}
 \lVert u_{N} \rVert_{\testY}^{2}&= \sum_{k=1}^{K} \frac{\vartriangle t^{k}}{2}(\lVert u_{N}^{k} \rVert_{V}^{2}+\lVert u_{N}^{k-1} \rVert_{V}^{2})+\lVert u_{N}^{0} \rVert_{H}^{2} \\
& =\sum_{k=1}^{K} \frac {\vartriangle t^{k}}{2}[\mathbf{u}_{N}^{k \ T} \mathbf{K}_{N} \mathbf{u}_{N}^{k}+\mathbf{u}_{N}^{k-1 \ T}\mathbf{K}_{N} \mathbf{u}_{N}^{k-1}]+\mathbf{u}_{N}^{0 \ T} \mathbf{M}_{N} \mathbf{u}_{N}^{0},
\end{align*}
where $\mathbf{K}_{N}:=[\langle \xi_{i},\xi_{j} \rangle_{V}]_{i,j=1}^{N} \in \mathbb{R}^{N \times N}$. Since in our case the reduced basis is orthonormal in $V$, $\mathbf{K}_{N}$ is the identity matrix. The operation count in the online phase, associated with computation of the residual norm and the spatio-temporal norm on $\testY$ is correspondingly $\mathcal{O}(Q_{R}^{2}K)$ and $\mathcal{O}(N^{2}(K+1))$.

We note that our a-posteriori error bound takes into account the error of the nonlinearity approximation \eqref{3.2.3}. It is given by
 \begin{align}\label{3.2.9}
\delta_{M}(\mu) = \max_{k \in K}\max_{x \in \Omega}|\nu_{M}(u_{N}^{k}(x);\mu) - \nu(u_{N}^{k}(x);\mu)|.
\end{align}
Since the EIM approximation $\nu_{M}(\cdot;\mu)$ is constructed out of truth solutions, we assume that $N$ is chosen in such a way that $\nu_{M}(u_{N}^{k}(x);\mu)\approx \nu_{M}(u_{\delta}^{k}(x);\mu)$. We note that \eqref{3.2.9} requires the knowledge of $\nu(u_{N}^{k}(\mu);x;\mu)$ and thus one full evaluation of the nonlinearity for all $K$ time steps. Therefore the certification procedure is not completely mesh-independent.

\section{Examples and numerical results}
\label{sec:4}
In this section we consider examples of quasilinear parabolic PDEs with strongly monotone differential operators and apply the proposed reduced-basis techniques to these problems.  

\subsection{\textbf{1-D magnetoquasistatic problem: analysis}}
\label{sec:4.1}
For the first numerical example we choose a 1-D magnetoquasistatic approximation of Maxwell's equations (see, e.g. \cite{bachinger2005numerical,salon1995finite}). Let $d=1$, $\Omega=(0,1)$ and $V:=H^{1}_{0}(\Omega) \hookrightarrow L^{2}(\Omega)=:H$. The norm on $V$ is $\lVert u\rVert_{V}^{2}:=\langle u',u' \rangle_{L^{2}}$, which is indeed a norm due to Poincare-Friedrichs inequality. We use the time interval $I=(0,0.2]$ and the parameter set $\parset:=[1, 5.5] \subset \mathbb{R}$. For a parameter $\mu \in \parset$, we want to find $u:=u(\mu)$ which solves
\begin{equation}\label{4.1.1}
\begin{aligned}[c]
\dot{u}-(\nu(|u'|;\mu)u')'& = g\\
u(t,x)&= 0\\
u_{o}(x)&=0
\end{aligned}
\begin{aligned}[c]
\quad&\text{on} \ I\times\Omega,\\ 
\quad &\forall \ (t,x) \in I\times \partial{\Omega},\\ 
\quad &\forall \ x \in \Omega. 
\end{aligned}
\end{equation}
We here used $g(x,t):=12\sin(2 \pi x)\sin(2 \pi t)$ and define $\nu(s;\mu)=\exp{(\mu s^2)}+1$ as the reluctivity function. 

We consider the quasilinear form for the weak formulation \eqref{2.1.5}, which here is given by
\begin{align}\label{4.1.2}
    a[u](u,v;\mu)=\int_{\Omega}\nu(|u'|;\mu)u'v'dx.
\end{align}
If the function $\nu(\cdot;\mu)\cdot : \mathbb{R}_{0}^{+} \rightarrow \mathbb{R}_{0}^{+}$ is strongly monotone, i.e. if
\begin{align}\label{4.1.3}
 (\nu(s_{2};\mu)s_{2}-\nu(s_{1};\mu)s_{1})(s_{2}-s_{1})\geq m_{a}(\mu)(s_{2}-s_{1}), \quad \forall s_{2}, s_{1} \in \mathbb{R}_{0}^{+} 
\end{align}
holds, then \eqref{4.1.2} satisfies the strong monotonicity condition \eqref{2.1.2}. Indeed, we set $s_{1}=w',  s_{2}=v'$ and integrating we get
\begin{align*}
    a[v](v,v-w)-a[w](w,v-w)&=\int_{\Omega} (\nu(v';\mu)v'-\nu(w';\mu)v')(v'-w')dx \\ \nonumber
    &\geq m_{a}(\mu) \int_{\Omega} (v'-w')^{2}dx=m_{a}(\mu) \lVert v-w \rVert_{V}^{2}. 
\end{align*}
It is clear that the reluctivity function $\nu(s;\mu)$ in our example satisfies \eqref{4.1.3}. Furthermore, the monotonicity constant can be taken as $m_{a}=\underset{\mu \in \parset}{\text{inf}}\ \underset{s \in \mathbb{R}_{+}}{\text{inf}} \ \nu(s;\mu)$, hence we have $m_{a}=2$ for our problem and the constant is parameter-independent. We also note that continuity of  $\nu(\cdot;\mu)$ implies hemicontinuity of \eqref{4.1.2} for every parameter $\mu \in \parset$. Thus the weak formulation \eqref{2.1.5} of the PDE \eqref{4.1.1} admits a unique solution. 

We specify the bilinear form $\langle A'(u;\mu)v,w\rangle_{V'V}=da[u](v,w;\mu)$ induced by the Fr\'echet derivative $A'(u;\mu) :V \times \parset \rightarrow V'$ of the nonlinear operator $A(u;\mu)$. It is then used to compute the Jacobian matrix \eqref{2.2.5} for Newton method. In the present example we have
 \begin{align*}
da[u](v,w;\mu)= \int_{\Omega}\left(2\mu \ \nu'(| u'|;\mu) u' +\nu(| u'|;\mu)\right)v'w' \ dx.
\end{align*}
The derivative for the reduced-basis scheme in the formula \eqref{3.1.8}, thanks to the chain rule, is given by
\begin{align*}
\partial_{1}\nu(|u_{N}'^{k}(x_{m}^{M})|;\mu)=2\mu\nu'(|u_{N}'^{k}(x_{m}^{M})|;\mu) u_{N}'^{k}(x_{m}^{M})\xi_{j}'(x_{m}^{M}),  
\end{align*}
where all the indices are according to \eqref{3.1.8}. However it was sufficient to drop the term  $\mathbf{E}_{N,M}(\mu)$ in $\mathbf{A}'_{N,M}(\mu)$ for our numerical experiments. This then corresponds to an inexact Newton-like method, which was applied in the numerical computations.

\subsection{\textbf{1-D magnetoquasistatic problem: numerical results}}
\label{sec:4.2}

The truth approximation is performed by the Petrov-Galerkin scheme, which is introduced in section \ref{sec:2}, where $V_{h}$ is the finite element space, composed of piecewise linear and continuous functions, defined on the partition of $\bar{\Omega}$ into $100$ equal subintervals and $\mathcal{N}_{h}=98$ nodes (excluding Dirichlet boundary nodes). For the time discretization we divide the interval $I$ into $K=200$ subintervals of length $\triangle t=10^{-3}$. We solve the problem with the Crank-Nicolson scheme \eqref{2.2.3}, while applying  Newton's method, described in section \ref{sec:2.2} on each time step for the numerical computation of the time snapshots. We iterate the Newton's method unless the norm of the residual \eqref{2.2.4} is less than the tolerance level, which we set to $10^{-8}$. 

We generate the RB-EIM model as follows: we start from $\parset_{train}^{EIM} \subset \parset$  (a uniform grid of size 200) and compute truth solutions for each parameter in $\parset_{train}^{EIM}$ to approximate the nonlinearity $\nu$  with its EIM counterpart $\nu_{M}$. We set $M_{max}=8$ as the maximal dimension of the EIM approximation space. Next we run the POD-Greedy procedure with $M=M_{\text{max}}$ and obtain $N_{\text{max}}=5$ for $\varepsilon_{RB}=10^{-5}$, where $\parset_{train}$ is a uniform grid over $\parset$ of size 400.  For the POD-Greedy procedure and method certification we use the estimator \eqref{3.2.2}. We solve the problem with the reduced Crank-Nicolson scheme \eqref{RB_scheme}, while applying RB Newton's method, described in section \ref{sec:3.2} on each time step for the numerical computation of the time snapshots. We iterate the Newton's method unless the norm of the residual \eqref{discrete_root} is less than the tolerance level, which we set to $10^{-8}$. 

Next we introduce a test sample $\parset_{test} \subset \parset$ of size 200 (uniformly random sample from $\parset$), the maximum of the estimator $\max\bigtriangleup_{N,M}:=\underset{\mu \in \parset_{test}}{\max}\bigtriangleup_{N,M}(\mu)$, the ``truth norm" error and its maximum over the test sample
\begin{align*}
\varepsilon^{true}_{N,M}(\mu):=\lVert u_{\delta}(\mu)-u_{N}(\mu) \rVert_{\testY}, \quad \max \varepsilon^{true}_{N,M}:= \underset{\mu \in \parset_{test}}{\max} \varepsilon^{true}_{N,M}(\mu).
\end{align*}
Once the reduced-basis model is constructed ($N_{\text{max}}=7,M_{\text{max}}=8)$, we verify the convergence with $N$  of $\max \bigtriangleup_{N,M}$ and $\max \varepsilon^{true}_{N,M}$ on a test sample $\parset_{test}$ and plot in Fig.\ref{fig:1} the $N$-$M$ convergence curves for different values of $M$. We can see that the estimator in Fig.\ref{fig:1}(b) reaches the desired tolerance level $\varepsilon_{RB}=10^{-5}$ for $(N_{\text{max}},M_{\text{max}})=(5,8)$. 

Next we investigate the influence of the EIM approximation error in the estimation process. We can split the estimator \eqref{3.2.2} into two parts: the reduced-basis and the nonlinearity approximation error estimation contributions
\begin{align*}
\bigtriangleup^{RB}_{N,M}(\mu):=\frac{1}{m_{a}}{\lVert R(\cdot;\mu)\rVert_{\testY'}} \ \text{and} \ \bigtriangleup^{EI}_{N,M}(\mu):=\frac{\delta_{M}(\mu)}{m_{a}}\lVert u_{N}(\mu) \rVert_{\testY}.
\end{align*}
We then set 
\begin{align}\label{4.2.1}
\bigtriangleup^{RB}_{N,M}:=\underset{\mu \in \parset_{test}}{\max} \bigtriangleup^{RB}_{N,M}(\mu), \quad \bigtriangleup^{EI}_{N,M}:=\underset{\mu \in \parset_{test}}{\max} \bigtriangleup^{EI}_{N,M}(\mu).
\end{align}
In Fig.\ref{fig:2}(a) we plot $\bigtriangleup^{RB}_{N,M}$  and $\bigtriangleup^{EI}_{N,M}$  for $1 \leq N \leq 5$ and $M=4$, $M=8$: we can see that $M$ has nearly no influence on $\bigtriangleup^{RB}_{N,M}$, but we observe the ``plateau" in $\bigtriangleup^{EI}_{N,M}$, which limits the convergence of the estimator \eqref{3.2.2} with increasing $N$. The separation points, or ``knees", of the $N$-$M$-convergence curves then reflect a (close-to) balanced contribution of both error terms.
 
\begin{figure}
\centering
\subfloat[]{\includegraphics[width=0.5 \textwidth]{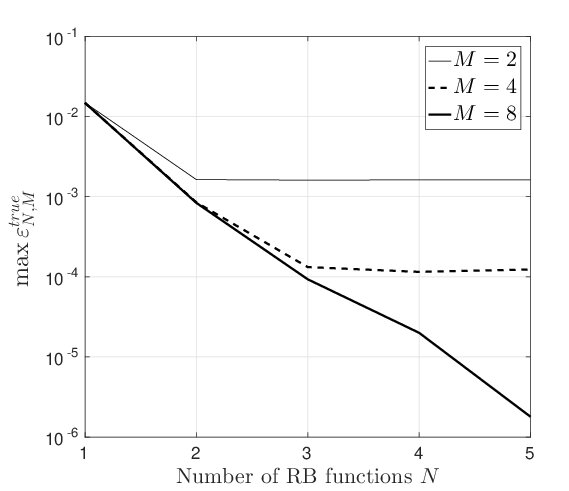}} 
\subfloat[]{\includegraphics[width=0.5\textwidth]{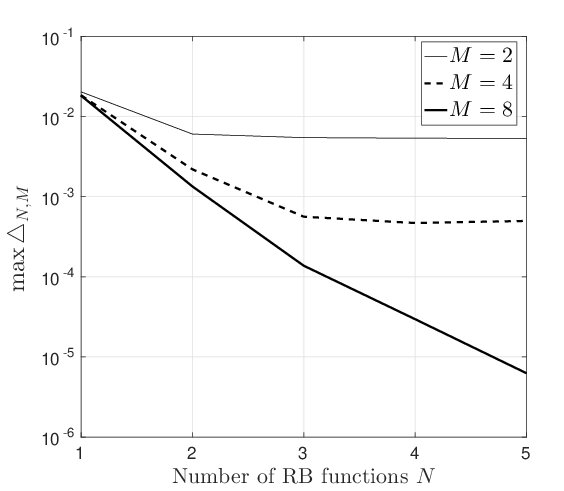}}
\caption{ (a): Convergence with $N$ of $\max \varepsilon^{true}_{N,M}$ for different values of $M$ on the test set, 1-D example. (b): Convergence with $N$ of $\max\bigtriangleup_{N,M}$ for different values of $M$ on the test set, 1-D example.}%
\label{fig:1}%
\end{figure}

\begin{figure}
\centering
\subfloat[]{\includegraphics[width=0.5\textwidth]{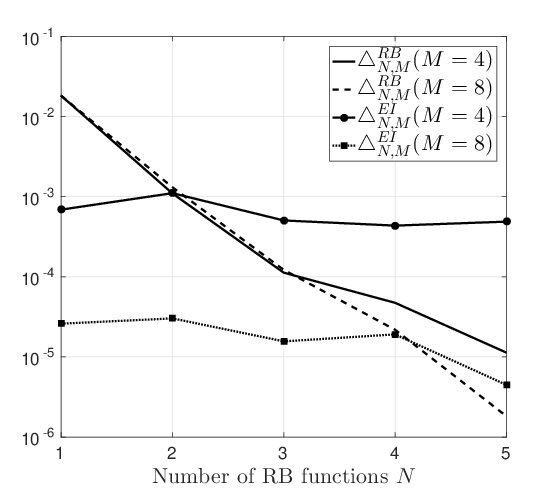}} 
\subfloat[]{\includegraphics[width=0.5 \textwidth]{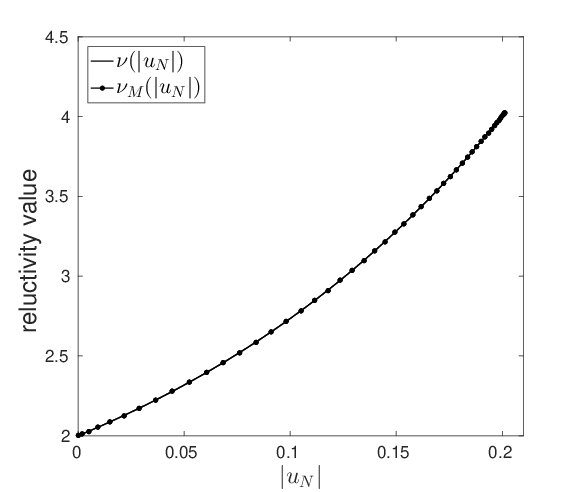}} 
\caption{ (a): The dependence of $\bigtriangleup^{RB}_{N,M}$ and $\bigtriangleup^{EI}_{N,M}$ contributions with $N$ for fixed values of $M$. (b): The reluctivity function $\nu(|u_{N}'(x)|;\mu)$ and its EI-approximation ($M=8$) $\nu_{M}(|u_{N}'(x)|;\mu)$ for the parameter $\mu=5.5$ at $t=0.2$ (b).}%
\label{fig:2}%
\end{figure}

In Table \ref{tab:1} we present, as a function of $N$ and $M$, the values of $\max\bigtriangleup_{N,M}$, $\bigtriangleup^{RB}_{N,M}$, $\bigtriangleup^{EI}_{N,M}$,  $\varepsilon^{true}_{N,M}$ and mean effectivities  $\bar{\eta}_{N,M}:=\frac{1}{|\parset_{test}|}\sum_{\mu \in \parset_{test}} \eta_{N,M}(\mu)$, where $\eta_{N,M}(\mu):=\bigtriangleup_{N,M}(\mu)/\lVert u_{\delta}(\mu)-u_{N}(\mu) \rVert_{\testY}$. We note that the tabulated $(N,M)$ values correspond roughly to the ``knees" of the $N$-$M$-convergence curves. We can see that the effectivities are lower bounded by 1 and are of moderate size, thus the error estimator is reliable and there is no significant overestimation of the true error. 

We then plot (see Fig. \ref{fig:2}(b)) the reluctivity function $\nu(|u_{N}'(x)|;\mu)$ and its EI approximation $\nu_{M}(|u_{N}'(x)|;\mu)$ for the parameter $\mu=5.5$ at $t=0.2$; we can see that there is no visible difference between the original function and its EIM counterpart. Although the problem at hand is merely chosen to illustrate the methodology, we report on the average CPU time for comparison. The finite element method takes $\approx 0.47$ sec to obtain the solution, and the RB method ($N_{\text{max}}, M_{\text{max}}$), which takes $\approx 0.08/0.10$ sec without and with the a-posteriori certification and results in the speed-up factor of 5.87/4.70\footnote{All the computations are performed in MATLAB on Intel Xeon(R) CPU E5-1650 v3, 3.5 GHz x 12 cores, 64 GB RAM}. The offline phase requires less than 10 minutes for our implementation. 
\begin{table}
\begin{tabular}{l l l l l l l l}
\hline\noalign{\smallskip}		
  $N$ & $M$ &$\max\bigtriangleup_{N,M}$ & $\bigtriangleup^{RB}_{N,M}$ & $\bigtriangleup^{EI}_{N,M}$ & $\max \varepsilon^{true}_{N,M}$ & $ \bar{\eta}_{N,M}$ \\
\noalign{\smallskip}\hline\noalign{\smallskip}
  2  & 2  & 6.10 E-03 & 5.60 E-03 & 7.60 E-04 & 1.60 E-03 & 4.00 \\
  3  & 4  & 5.62 E-04 & 5.05 E-04 & 1.12 E-04 & 1.32 E-04 & 5.82 \\
  5  & 8  & 6.25 E-06 & 4.47 E-06 & 1.81 E-06 & 1.79 E-06 & 4.58 \\
\noalign{\smallskip}\hline
\end{tabular}
\caption{Performance of the 1D RB-EIM magnetoquasistatic approximation of Maxwell's equations on the test set}
\label{tab:1} 
\end{table}

\subsection{\textbf{2-D magnetoquasistatic problem: analysis}}
\label{sec:4.3}

As second example we consider a 2-D magnetoquasistatic problem for modelling of eddy currents in a steel pipe\footnote{http://www.femm.info/wiki/TubeExample}. Let $\bar{\Omega}=\bar{\Omega}_{1}\bigcup \bar{\Omega}_{2}$ be a circular cross-section of the steel pipe with radius $r_{2}$, where $\Omega_{1}$ is the conducting domain (iron) and $\Omega_{2}$ is the non-conducting domain of radius $r_{1}$. The wire is represented by the part with the radius $r_{0}$ and the complementary part is the air gap (see Fig.\ref{fig:3}(a)). We assume that the magnetic reluctivity function and the electric conductivity function have different structure on conducting and non-conducting domains, respectively, i.e.
\begin{equation*}
\begin{aligned}[c]
  \nu(x,s) =
  \begin{cases}
                                   \nu_{1}(s), \  \text{for $x \in \Omega_{1}$}, \\
                                   \nu_{2},  \  \text{for $x \in \Omega_{2}$} \\
  \end{cases}\ \text{and}
\end{aligned}
\quad
\begin{aligned}[c]
  \sigma(x) =
  \begin{cases}
                                   \sigma_{1}>0, \  \text{for $x \in \Omega_{1}$}, \\
                                   \epsilon>0,  \  \text{for $x \in \Omega_{2}$}, \\
  \end{cases}
\end{aligned}
\end{equation*}
where $\nu_{2},\sigma_{1}>0$ denote constants. We assume that the reluctivity function satisfies 
\begin{align}\label{4.3.1}
   0 < \nu_{\text{LB}} \leq \nu(x,s)\leq \nu_{UB}, \quad  \forall x\in \Omega, \ s \in \mathbb{R}^{+}_{0},
\end{align}
where $\nu_{\text{LB}}$ and $\nu_{UB}$ are accessible constants. We note that the air-gap and the coils in the steel pipe are electrically non-conductive, i.e. $\sigma(\xi)=0$ for $\xi \in \Omega_{2}$. However, we introduce a regularization parameter $\epsilon=10^{-8}$ as a value of $\sigma$ for the non-conducting domain. This allows us to consider a pure parabolic problem instead of a parabolic-elliptic system with differential-algebraic structure (see, e.g. \cite{Stykel2017}). We set $\mu:=\sigma_{1}$  and define the parameter set $\parset=[5\cdot 10^6, 10^7]$ and the time interval $I=(0,0.02]$. We thus have a parametrized quasilinear parabolic equation
\begin{equation}\label{4.3.2}
\begin{aligned}[c]
\sigma(x;\mu)\dot{u}-\nabla \cdot(\nu(x,\lvert \nabla u\rvert)  \nabla  u)&=g \quad  \\
u(t,x)&= 0\\
u_{o}(x)&=0
\end{aligned}
\begin{aligned}[c]
\quad&\text{on} \ I\times\Omega,\\ 
\quad &\forall \ (t,x) \in I\times \partial{\Omega},\\ 
\quad &\forall \ x \in \Omega. 
\end{aligned}
\end{equation}
The right-hand side is the electric-flux density
\begin{equation*}
  g(x,t) =
  \begin{cases}
                                   \frac{I_{e}(t)}{2\pi r_{0}}, \  \text{for $x \in \Omega_{1}$}, \\
                                   0,  \  \text{for $x \in \Omega_{2}$}, \\
  \end{cases}
\end{equation*}
where $I_{e}(t)=100\cdot \sin(100\pi t)$ is the electric current. 

We consider the quasilinear form for the weak formulation \eqref{2.1.5}, which here is given by
\begin{align}\label{4.3.3}
    a[u](u,v;\mu)=\int_{\Omega}\nu(x,|\nabla u|;\mu)\nabla u \cdot \nabla v \ dx.
\end{align}
In practical applications, the nonlinear reluctivity function is often defined through magnetization curves or $|B|$-$|H|$ curves. The underlying physical properties of ferromagnetic materials determine the $|B|$-$|H|$ curve. These curves are naturally strongly monotone and, in practice, their analytical form is unknown. Instead, only a finite number of discrete points $(|H_{k}|,|B_{k}|), k=1,...,K_{c}$ with $|H_{k}|$,$|B_{k}|$  denoting the magnitude of the magnetic field (measured in ampere/meter) and magnetic flux (measured in tesla), is given from the real life measurements. In order to reconstruct a continuous, monotone $|B|$-$|H|$ curve, monotonicity-preserving interpolation with cubic splines is applied \cite{heise1994analysis}. We define a mapping $g_{1}:\mathbb{R}_{0}^{+}\rightarrow \mathbb{R}_{0}^{+}$ which determines the magnetization curve via $|B|=g_{1}(|H|)$. An example of a $|B|$-$|H|$ curve, based on the measurements of a ferromagnetic material, which is used in our problem, is given in Fig.\ref{fig:3}(b); the real life measurements were provided by \cite{Schops2020}. The mapping $s \mapsto \nu_{1}(s)s, s\in \mathbb{R}_{+}$ then denotes the inverse $g_{1}^{-1}$ of $g_{1}$ and thus also is strongly monotone. The nonlinear reluctivity function $\nu_{1}:\mathbb{R}_{0}^{+}\rightarrow \mathbb{R}_{0}^{+}$ then is given by $\nu_{1}(s):=g_{1}^{-1}(s)/s,s\in \mathbb{R}_{+}$; it is required that $\nu_{1} \in C^{1}(\mathbb{R}^{+}_{0};\mathbb{R}^{+})$ and the spline approximation technique guarantees this property. If $g_{1}^{-1}(s)=\nu_{1}(s)s$ satisfies the strong monotonicity condition \eqref{4.1.3}, then the mapping $\mathbf{s} \mapsto \nu_{1}(|\mathbf{s}|)\mathbf{s}, \mathbf{s} \in \mathbb{R}^{2}$ is strongly monotone with monotonicity constant $\nu_{LB}$ and Lipschitz continuous with Lipschitz constant $\nu_{UB}$. The mapping $\mathbf{s}\mapsto \nu_{2}\mathbf{s},  \mathbf{s} \in \mathbb{R}^{2}$ is linear, therefore $\mathbf{s}\mapsto \nu(x,|\mathbf{s}|)\mathbf{s}$ is strongly monotone for all $x \in \Omega$. The form \eqref{4.3.3} then is strongly monotone with the monotonicity constant $\nu_{LB}$ and Lipschitz continuous with the Lipschitz constant $3\nu_{UB}$ (see \cite{heise1994analysis} for the corresponding proofs). Hence the weak formulation \eqref{2.1.5} of the PDE \eqref{4.3.2} admits a unique solution.

\begin{figure}
\centering
\subfloat[]{\includegraphics[width=0.48\textwidth]{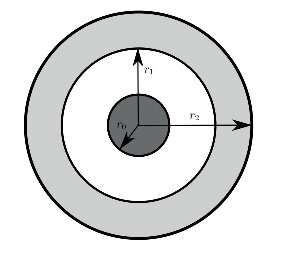}} 
\subfloat[]{\includegraphics[width=0.50 \textwidth]{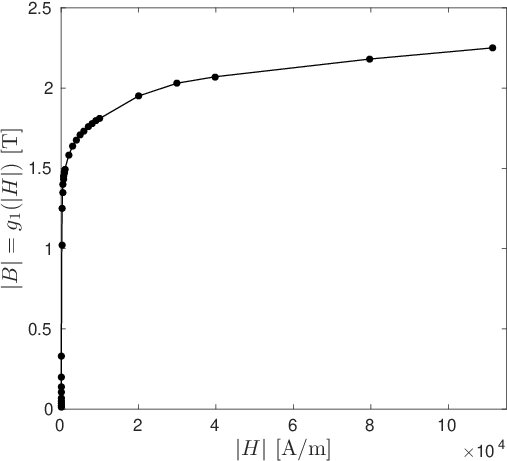}} 
\caption{ (a): Geometry of the computational domain: the wire (dark grey), the air gap (white), the iron (bright grey). (b): Example of a $|B|$-$|H|$ curve, approximated with cubic splines from measured data points.}%
\label{fig:3}%
\end{figure}

We specify the bilinear form $\langle A'(u;\mu)v,w\rangle_{V'V}=da[u](v,w;\mu)$ induced by the Fr\'echet derivative $A'(u;\mu) :V \times \parset \rightarrow V'$ of the nonlinear operator $A(u;\mu)$. It is then used to compute the Jacobian matrix \eqref{2.2.5} for Newton method. With
\begin{equation*}
n[u] =
  \begin{cases}
                                   \frac{\nabla u}{|\nabla u|}, \  \text{for} \  \nabla u \neq 0, \\
                                   0,  \  \text{for} \ \nabla u = 0, \\
  \end{cases}
\end{equation*}
we have
 \begin{align*}
da[u](v,w;\mu)=\int_{\Omega}\nu'(x,|\nabla u|;\mu)(n[u] \cdot  \nabla w)(\nabla u \cdot \nabla v)+\nu(x,|\nabla u|;\mu)\nabla v \cdot \nabla w \ dx,
\end{align*}
and the derivative for the reduced-basis scheme in the formula \eqref{3.1.8}, thanks to the chain rule, is given by
\begin{align*}
\partial_{1}\nu(x;|\nabla u_{N}^{k}(x_{m}^{M})|;\mu)=\nu'(x;|\nabla u_{N}^{k}(x_{m}^{M})|;\mu) n[u_{N}^{k}](x_{m}^{M})\cdot\nabla\xi_{j}(x_{m}^{M}),  
\end{align*}
where all the indices are according to \eqref{3.1.8}. 

In this example, the monotonicity constant $m_{a}(\mu)$ is not available analytically. As it was mentioned earlier in the discussion on $|B|$-$|H|$ curves, we can choose $\nu_{LB}>0$ as our monotonicity constant.  However, since for each parameter $\mu \in \parset$ there holds
\begin{align}\label{4.3.4}
m_{a}(\mu):=\underset{k \in K}{\min} \ \underset{x \in \Omega}{\min} \ \nu_{1}(|\nabla u_{N}^{k}(x)|;\mu)\geq \nu_{LB},
\end{align}
and the computation of \eqref{4.3.4} only requires one full evaluation of the nonlinearity, which already has been performed to evaluate \eqref{3.2.9}, we here use $m_{a}(\mu)$ as our constant for the estimation.

\subsection{\textbf{2-D magnetoquasistatic problem: numerical results}}
\label{sec:4.4}
The truth approximation is performed by the Petrov-Galerkin scheme, which is introduced in section \ref{sec:2}, where $V_{h}$ is the finite element space, composed of piecewise linear and continuous functions, defined on a triangle mesh containing $4374$ triangles and $\mathcal{N}_{h}=2107$ nodes (excluding Dirichlet boundary nodes). For the time discretization we divide the interval $I$ into $K=200$ subintervals of length $\triangle t=10^{-4}$. The nonlinear reluctivity function $\nu_{1}$ is reconstructed from the real $B-H$ measurements using monotonicity-preserving cubic spline interpolation and  $\nu_{2}$ value is chosen as the reluctivity of air. We then solve the problem with the Crank-Nicolson scheme \eqref{2.2.3}, while applying Newton's method, described in section \ref{sec:2.2}, on each time step for the numerical computation of the time snapshots. We iterate the Newton's method unless the norm of the residual \eqref{2.2.4} is less than the tolerance level, which we set to $10^{-8}$.

We generate the RB-EIM model as follows: we start from $\parset_{train}^{EIM} \subset \parset$  (a uniform grid of size 200) and compute truth solutions for each parameter in $\parset_{train}^{EIM}$ to approximate the nonlinearity $\nu_{1}$  with the EIM counterpart $\nu_{1}^{M}$. We set $M_{\text{max}}=44$ as the maximal dimension of the EIM approximation space. Next we run the POD-Greedy procedure with $M=M_{\text{max}}$ and obtain $N_{\text{max}}=14$ for $\varepsilon_{RB}=10^{-4}$, where $\parset_{train}$ is a uniform grid over $\parset$ of size 400.  For the POD-Greedy procedure and the method certification we use the estimator \eqref{3.2.2}. The monotonicity constant is evaluated as in \eqref{4.3.4}. We solve the problem with the reduced Crank-Nicolson scheme \eqref{RB_scheme}, while applying RB Newton's method, described in section \ref{sec:3.2}, on each time step for the numerical computation of the time snapshots. We iterate the Newton's method unless the norm of the residual \eqref{discrete_root} is less than the tolerance level, which we set to $10^{-8}$. 

Then we verify the convergence with $N$ of $\max \varepsilon^{true}_{N,M}$ (Fig. \ref{fig:4}(a)) and $\max \bigtriangleup_{N,M}$ (Fig. \ref{fig:4}(b)) on a test sample $\parset_{test}$ (a uniformly random sample of size 200) for different values of $M$. We can see  that the estimator in Fig.\ref{fig:3}(b) reaches the desired tolerance level $\varepsilon_{RB}=10^{-4}$ for $(N_{\text{max}},M_{\text{max}})=(14,44)$. We note that the convergence is not monotone at some points due to the EIM interpolation of the non-polynomial nonlinearity behind the problem. We can also see from Fig.\ref{fig:4}(a) that increasing $M$ above 20 has nearly no impact on the convergence of the truth norm error, but the estimator in Fig.\ref{fig:4}(b) still shows a considerable decrease with increasing $M$. Indeed, in Fig.\ref{fig:5}(a) we plot $\bigtriangleup^{RB}_{N,M}$  and $\bigtriangleup^{EI}_{N,M}$ as defined in \eqref{4.2.1} for $1 \leq N \leq 14$ and $M=20$, $M=44$: we can see that $M$ has nearly no influence on $\bigtriangleup^{RB}_{N,M}$, but we can observe the ``plateau" in $\bigtriangleup^{EI}_{N,M}$, which limits the convergence of the estimator \eqref{3.2.2} with increasing $N$. We also plot the values of $\bigtriangleup_{N,M}(\mu)$ and $\varepsilon^{true}_{N,M}(\mu)$ and the truth error for $(N_{\text{max}},M_{\text{max}})$ for every parameter $\mu \in \parset_{test}$ in Fig.\ref{fig:5}(b).

\begin{figure}
\centering
\subfloat[]{\includegraphics[width=0.49\textwidth]{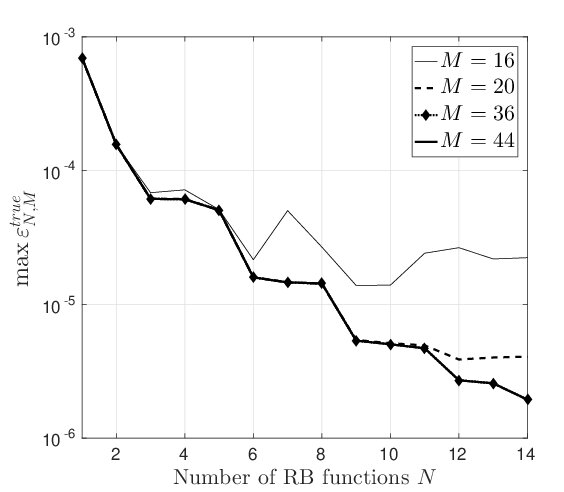}} 
\subfloat[]{\includegraphics[width=0.51 \textwidth]{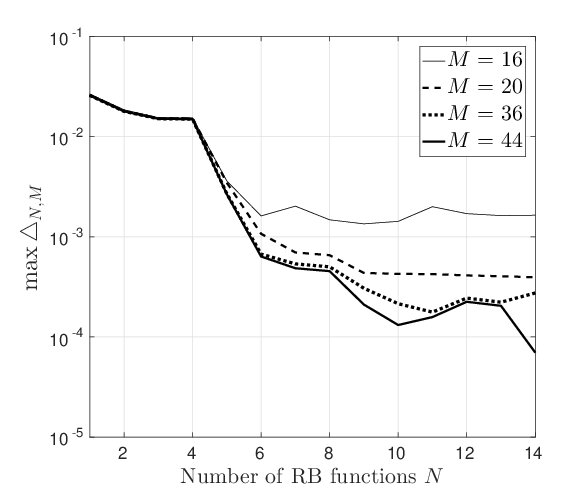}}
\caption{(a): Convergence with $N$ of $\max \varepsilon^{true}_{N,M}$ for different values of $M$ on the test set, 2-D example. (b): Convergence with $N$ of $\max\bigtriangleup_{N,M}$ for different values of $M$ on the test set, 2-D example.}%
\label{fig:4}%
\end{figure}

\begin{figure}
\centering
\subfloat[]{\includegraphics[width=0.5\textwidth]{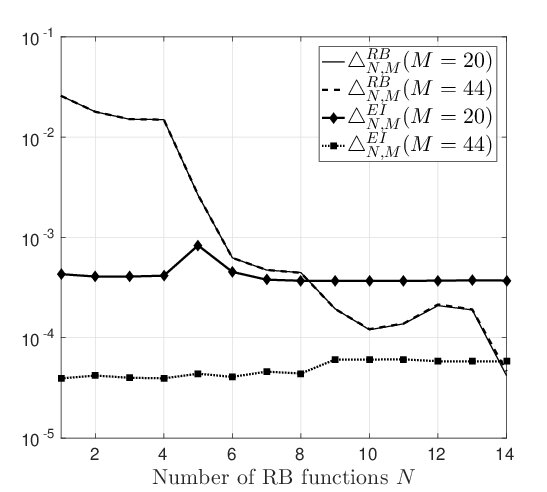}} 
\subfloat[]{\includegraphics[width=0.5 \textwidth]{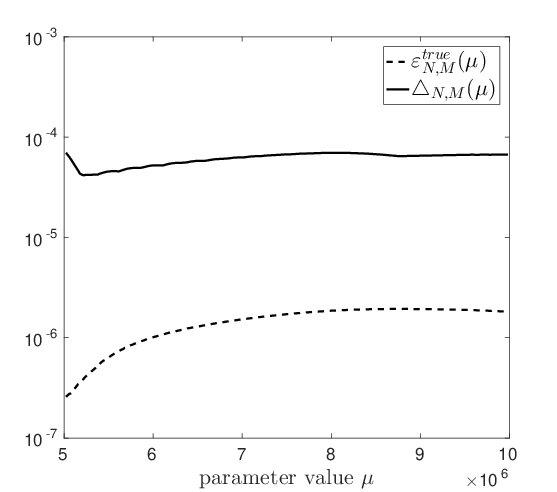}}
\caption{(a): The dependence of $\bigtriangleup^{RB}_{N,M}$ and $\bigtriangleup^{EI}_{N,M}$ contributions with $N$ for fixed values of $M$. (b): Values of $\varepsilon^{true}_{N,M}$ and $\max \vartriangle_{N,M}$ for $(N_{\text{max}},M_{\text{max}})=(14,44)$ on the test set.}
\label{fig:5}%
\end{figure}

\begin{figure}
\centering
\subfloat[]{\includegraphics[width=0.42\textwidth]{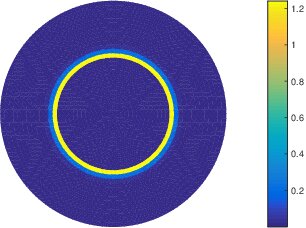}} 
\qquad
\subfloat[]{\includegraphics[width=0.42\textwidth]{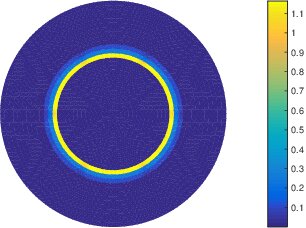}} \\
\subfloat[]{\includegraphics[width=0.42\textwidth]{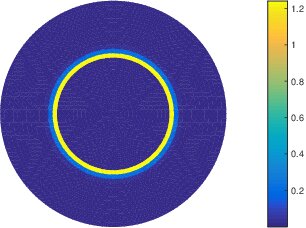}}  
\qquad
\subfloat[]{\includegraphics[width=0.42\textwidth]{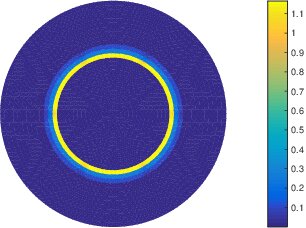}}\\
\caption{The truth magnetic flux density $|\nabla u_{\delta}|$ for $\mu=10^7$  at (a) $t=0.01$, (b) $t=0.02$. The reduced-basis magnetic flux density $|\nabla u_{N}|$ for $\mu=10^7$ at (c) $t=0.01$, (d) $t=0.02$.} 
\label{fig:6}
\end{figure}

In Table \ref{tab:2} we present, as a function of $N$ and $M$, the values of $\max\bigtriangleup_{N,M}$, $\bigtriangleup^{RB}_{N,M}$, $\bigtriangleup^{EI}_{N,M}$,  $\max \varepsilon^{true}_{N,M}$ and the mean effectivities $\bar{\eta}_{N,M}$. We note that the tabulated $(N,M)$ values correspond roughly to the ``knees" of the $N$-$M$-convergence curves (see example 1 for the terminology and definitions). We can see that the effectivities are lower bounded by 1, but the values are relatively large. We conject that this is related to the structure of the nonlinearity and the effectivities are proportional to $C\cdot \nu_{UB}/ \nu_{LB}$, where $C$ is some constant.

\begin{table}
\begin{tabular}{l l l l l l l l l}
\hline\noalign{\smallskip}		
  $N$ & $M$ &$\max\bigtriangleup_{N,M}$ & $\bigtriangleup^{RB}_{N,M}$ & $\bigtriangleup^{EI}_{N,M}$ & $\max\varepsilon^{true}_{N,M}$ & $ \bar{\eta}_{N,M}$  \\
\noalign{\smallskip}\hline\noalign{\smallskip}
  6  & 16 & 1.60 E-03 & 6.68 E-04 & 1.40 E-03 & 2.15 E-05 &  98.06 \\
  9  & 20 & 4.34 E-04 & 1.94 E-04 & 3.67 E-04 & 5.40 E-06 &  89.84  \\
  11 & 36 & 1.76 E-04 & 1.38 E-04 & 1.04 E-04 & 4.68 E-06 &  64.28  \\
  14 & 44 & 6.97 E-05 & 4.63 E-05 & 5.81 E-05 & 1.93 E-06 &  48.27  \\
\noalign{\smallskip}\hline
\end{tabular}
\caption{Performance of 2-D RB-EIM model on the test set}
\label{tab:2} 
\end{table}

In Fig.\ref{fig:6} we show the truth finite element magnetic flux density $|\nabla u_{\delta}(x,t,\mu)|$ and the corresponding reduced magnetic flux density $|\nabla u_{N}(x,t,\mu)|$ for $\mu=10^7$ and $t=0.01$ and $t=0.02$. We observe that flux densities look very similar. Next we compare the average CPU time required for both the finite element method, which takes $\approx 70$ sec to obtain the solution, and the RB method with $(N_{\text{max}}, M_{\text{max}})=(14,44)$, which takes $\approx 1.80/2.42$ without and with the a-posteriori certification and results in the speed-up factors (rounded) of 39 and 29, respectively. The offline phase requires the knowledge of the truth finite-element solutions for the EIM approximation step. Since 200 truth solutions were generated in the consecutive order, it takes $\approx$ 4 hours. The generation of these truth solutions could be performed in parallel which would reduce the offline time. The POD-Greedy sampling takes $\approx$ 40 minutes for our implementation.

\section{Conclusion}
\label{sec:5}
In this paper we propose the space-time reduced-basis method for quasilinear parabolic PDEs. We think that our space-time formulation combined with the chosen Petrov-Galerkin discretization provides an elegant approach to treat these kind of problems. We present a new a-posteriori error bound and use it for the reduced basis construction with the POD-Greedy procedure.  The developed methodology is applied to the magnetoquasistatic approximation of Maxwell's equations and numerical results confirm a good speed-up factor, which supports the validity of this approach. The reduced-basis methods developed in the paper will further be extended to treat more complicated industrial problems. It will further have a significant impact on the PASIROM project\footnote{http://www.pasirom.de/}, where the surrogate reduced-basis models are planned to be used in the optimization of electrical machines.

\section*{Acknowledgements}
Both authors acknowledge the support of the collaborative research project PASIROM funded
by the German Federal Ministry of Education and Research (BMBF) under grant no. 05M2018.





\end{document}